\documentclass[11pt]{article}
\usepackage[utf8]{inputenc}
\usepackage[english]{babel}
\usepackage{amssymb,amsmath}
\usepackage[usenames, dvipsnames]{color}
\usepackage{amsthm}
\usepackage{cite}
\let\OLDthebibliography\thebibliography
\renewcommand\thebibliography[1]{
  \OLDthebibliography{#1}
  \setlength{\parskip}{0pt}
  \setlength{\itemsep}{0pt plus 0.3ex}
}
\newtheorem{theorem}{{\scshape Theorem}}[section]
\newtheorem{lemma}[theorem]{{\scshape Lemma	}}
\newtheorem{corollary}[theorem]{{\scshape Corollary}}
\newtheorem{proposition}[theorem]{{\scshape Proposition}}

\begin{document}
\title{Comparison of addition  and  multiplication in a skew brace}
\author{Baojun Li, Timur Nasybullov, Vyacheslav Zadvornov}
\date{}
\maketitle
\begin{abstract}
A.~Smoktunowicz and L.~Vendramin conjectured that if $A=(A,\oplus,\odot)$ is a finite skew brace with solvable additive group $A_{\oplus}$, then the multiplicative group $A_{\odot}$ of $A$ is also solvable. Proving or disproving this conjecture is currently an open problem. 

The interest to the conjecture of A.~Smoktunowicz and L.~Vendramin is due to the fact that, despite the fact that the addition and multiplication in a skew brace are related to each other, they can be very different. The present work focuses on comparing addition and multiplication in a skew brace. The results presented in the paper say that if $B$ is a characteristic subgroup of $A_{\oplus}$, then under certain conditions on elements $a,b\in A$ the images of $a\odot b$ and $a\oplus b$ coincide in $A_{\oplus}/B$.

As a corollary we conclude that if $A$ is a finite skew brace such that the derived subgroup $A_{\oplus}^{\prime}$ is cyclic, then $A_{\odot}$ is solvable. This statement gives a positive answer to the conjecture of A.~Smoktunowicz and L.~Vendramin in the case when $A_{\oplus}^{\prime}$ is a cyclic group.

~
\end{abstract}
\section{Introduction}
A skew brace $A=(A,\oplus,\odot)$ is an algebraic system with two binary algebraic operations  $\oplus$, $\odot$ such that $A_{\oplus}=(A,\oplus)$, $A_{\odot}=(A,\odot)$ are groups and the equality
\begin{equation}\label{mainleft}
a\odot(b\oplus c)=(a\odot b)\ominus a \oplus (a\odot c)
\end{equation}
holds for all $a,b,c\in A$, where $\ominus a$ denotes the inverse to $a$ element with respect to the operation~$\oplus$ (we denote by $a^{-1}$ the inverse to $a$ element with respect to $\odot$). The group $A_{\oplus}$ is called the additive group of a skew brace $A$, and the group $A_{\odot}$ is called the multiplicative group of a skew brace $A$. If $A_{\oplus}$ is abelian, then $A$ is called a brace.

Braces were introduced by Rump in \cite{Rum} in order to study non-degenerate involutive
set-theoretic solutions of the Yang-Baxter equation. Skew braces were introduced by Guarnieri and Vendramin in \cite{GuaVen} in order to study non-degenerate   
set-theoretic solutions of the Yang-Baxter equation which are not necessarily involutive. Skew braces have connections with other algebraic structures such as groups with exact factorizations, Zappa-Sz\'{e}p products,
triply factorized groups and Hopf-Galois extensions  \cite{SmoVen}. Some algebraic aspects of skew braces are studied in \cite{Chi, CedSmoVen, GorNas, Jes, Nas, NasNov, SmoVen, KonSmoVen,Trap,Rum, DamMas}. Thanks to its relations with the Yang-Baxter equation skew braces can be applied for constructing representations of virtual braid groups and invariants of virtual links\cite{BarNasmult1, BarNasmult2, BarNasmult3}. A big list of problems concerning skew braces is collected in \cite{Ven}.

Equality (\ref{mainleft}) makes the additive group $A_{\oplus}$ and the multiplicative group $A_{\odot}$ of a skew brace $A$ strongly connected with each other. The following problems formulated in \cite[Problem 19.90]{kt} tell about some of such connections.\medskip

1.  Does there exist a skew brace $A$ with nilpotent group $A_{\odot}$ but non-solvable group $A_{\oplus}$?

2.  Does there exist a finite skew brace $A$ with nilpotent group $A_{\odot}$ but non-solvable group $A_{\oplus}$?

3.  Does there exist a skew brace $A$ with solvable group $A_{\oplus}$ but non-solvable group $A_{\odot}$?

4.  Does there exist a  finite skew brace $A$ with solvable group $A_{\oplus}$ but non-solvable group $A_{\odot}$?\medskip

\noindent The first and the second problems from the list above were studied, for example, in \cite{Byo, SmoVen, pr2doo, Nas, Trap}. The most general known results about the first problem concern two-sided skew braces. In \cite{Nas} it is proved that if $A=(A,\oplus,\odot)$ is a two-sided skew brace such that $A_{\odot}$ is nilpotent of class $k$, then $A_{\oplus}$ is solvable of derived length at most $2k$. In \cite{Trap} this result is refined. It is proved that if $A=(A,\oplus,\odot)$ is a two-sided skew brace such that $A_{\odot}$ is solvable of derived length $k$, then $A_{\oplus}$ is solvable of derived length at most $k+1$. The complete answer to the second problem is obtained in \cite{pr2doo}, where it is proved that every finite skew brace with nilpotent multiplicative group has a solvable additive group. The third problem from the list above is completely solved in \cite{Nas}, where an example of an infinite skew brace with solvable additive group and non-solvable multiplicative group is constructed. The fourth problem from the list above was studied, for example, in \cite{Nas,SmoVen,GorNas}. One of the first results in this direction is introduced in \cite{SmoVen}. It states that if the additive group $A_{\oplus}$ of a finite skew brace $A$ is nilpotent, then $A_{\odot}$ is solvable. In \cite{Nas} the fourth problem from the list above is solved for two-sided skew braces. It is proved there that if the additive group $A_{\oplus}$ of a finite two-sided skew brace $A$ is solvable, then $A_{\odot}$ is solvable. The last known result in this direction is proved in \cite{GorNas}. It says that if $A$ is a finite skew brace such that $|A|$ is not divisible by $3$ and $A_{\oplus}$ is solvable, then $A_{\odot}$ is solvable.
 
The interest to the four problems formulated above is due to the fact that, despite the fact that the addition and multiplication in a skew brace are related to each other, they can be very different. The present work focuses on comparing addition and multiplication in a skew brace. The results presented in the paper say that if $B$ is a characteristic subgroup of $A_{\oplus}$, then under certain conditions on elements $a,b\in A$ the images of $a\odot b$ and $a\oplus b$ coincide in $A_{\oplus}/B$. There are several similar results in the paper. A typical one is the following statement\medskip

\noindent \textbf{Proposition \ref{firstproposition}.} \textit{Let $A$ be a skew brace, $B$ be a characteristic subgroup of $A_{\oplus}$, and $W$ be a set of group words. For $x\in A$ denote by $\overline{x}$ the image of $x$ under the natural homomorphism $A_{\oplus}\to A_{\oplus}/B$. If $W({\rm Aut}(A_{\oplus}/B))=1$, then for all $a,b\in W(A_{\odot})$ the equalities $\overline{a\odot b}=\overline{a\oplus b}$, $\overline{a^{-1}}=\overline{\ominus a}$ hold.}\medskip

As a corollary from the obtained results we prove the following theorem which gives a positive answer to the fourth problem from the list above  in the case when $A_{\oplus}^{\prime}$ is a cyclic group.\medskip

\noindent\textbf{Theorem \ref{maintheorem}.} \textit{Let $A$ be a finite skew brace such that $A_{\oplus}^{\prime}$ is cyclic. Then $A_{\odot}$ is solvable.}\medskip

The obtained theorem is not covered by the known results in this direction. If a group has a cyclic derived subgroup, then it is not necessarily nilpotent, so, the result from \cite{SmoVen} doesn't cover Theorem~\ref{maintheorem}. Skew brace $A$ from the formulation of Theorem~\ref{maintheorem} doesn't have to be two-sided, so, the result from \cite{Nas} doesn't cover Theorem~\ref{maintheorem}. Finally, the order of $A$ in Theorem~\ref{maintheorem} can be divisible by $3$, for example, when $A_{\oplus}^{\prime}=\mathbb{Z}_{3^n}$, so, the result from \cite{GorNas} doesn't cover Theorem~\ref{maintheorem}. 

The authors of \cite{DamMas} study finite groups in which every proper characteristic subgroup is cyclic, and skew braces which have such groups as additive groups. If $A$ is a finite skew brace such that every proper characteristic subgroup of $A_{\oplus}$ is cyclic, then either $A_{\oplus}^{\prime}$ is cyclic or $A_{\oplus}^{\prime}=A_{\oplus}$. So, Theorem~\ref{maintheorem} applies to such skew braces. 

~\\
\textbf{Acknowledgment.} The first and the second author were supported by a National Natural Science Foundation of China (12371021).

\section{Preliminaries}
In this section we give the necessary definitions, fix notations and prove some preliminary results. If $G$ is a group and $a,b\in G$, then the commutator $[a,b]$ is defined as $[a,b]=a^{-1}b^{-1}ab$. For subgroups $A,B$ of $G$ we denote by $[A,B]=\langle [a,b]~|~a\in A, b\in B\rangle$. The derived series 
$$G^{(1)}\geq G^{(2)}\geq G^{(3)}\geq \dots$$ 
of $G$ is inductively defined by the formulas  $G^{(1)}=G$, $G^{(i+1)}=[G^{(i)},G^{(i)}]$ for $i=1,2,3,\dots$ The derived subgroup $G^{(2)}$ is also denoted by $G^{\prime}$. If $H$ is a subgroup of $G$, then the centralizer of $H$ in $G$ is denoted by $C_G(H)$. The center of $G$ is denoted by $Z(G)$.

Let $w(x_1,x_2,\dots,x_n)$ be an element from the free group $F_n$ with the free generators $x_1,x_2,\dots,x_n$. This element is called a group word in letters $x_1,x_2,\dots,x_n$. For a group $G$ and elements $a_1,a_2,\dots,a_n\in G$ one can define the element $w(a_1,a_2,\dots,a_n)\in G$ writing $a_i$ instead of $x_i$ in $w$ for all $i=1,2,\dots,n$. The subgroup of $G$ generated by elements $w(a_1,a_2,\dots,a_n)\in G$ for all elements $a_1,a_2,\dots,a_n\in G$ is denoted by $w(G)$. If $W$ is a set of group words, and $G$ is a group, then the subgroup of $G$ generated by $w(G)$ for all $w\in W$ is denoted by $W(G)$. This subgroup is called the verbal subgroup of $G$ generated by the set of group words $W$. This subgroup is characteristic in $G$. All members of the derived series of a group are verbal subgroups of this group.

Let $A$ be a skew brace.  Equality (\ref{mainleft}) implies that the unit element of $A_{\oplus}$ coincides with the unit element of $A_{\odot}$. We denote this element by $0$. 
For $a\in A$ the map $\lambda_a: x\mapsto \ominus a\oplus (a\odot x)$ is an automorphism of the group $A_{\oplus}$. Moreover, the map $a\mapsto \lambda_a$ gives a homomorphism $A_{\odot}\to {\rm Aut}(A_{\oplus})$. From the definition of $\lambda_a$ it follows that for all $a,b\in A$ the equality
\begin{equation}\label{multsum}
a\odot b=a\oplus \lambda_a(b)
\end{equation}
holds. This equality implies the equality 
\begin{equation}\label{inverseoposit}
a^{-1}=\ominus\lambda_{a^{-1}}(a).
\end{equation}
The following result follows from formulas (\ref{multsum}), (\ref{inverseoposit}).
\begin{lemma}\label{easylemma} Let $A$ be a skew brace. If $B$ is a characteristic subgroup of $A_{\oplus}$, then $B$ is a subbrace of $A$.
\end{lemma}
\begin{proof} It is necessary to prove that for all $a,b\in B$ the elements $a\odot b$ and $a^{-1}$ belong to $B$. The fact that the element $a\odot b$ belongs to $B$ follows from formula (\ref{multsum}) since $B$ is characteristic in $A_{\oplus}$ and $\lambda_a$ is an automorphism of~$A_{\oplus}$. The fact that $a^{-1}$ belongs to $B$ follows from formula (\ref{inverseoposit}) since $B$ is characteristic in $A_{\oplus}$, and $\lambda_{a^{-1}}$ is an automorphism of $A_{\oplus}$.
\end{proof}

If $A$ is a skew brace, and $a,b\in A$, then we denote by $[a,b]_{\oplus}$, $[a,b]_{\odot}$ respectively the additive commutator and the multiplicative commutator of $a,b$:
\begin{align*}
[a,b]_{\oplus}=\ominus a\ominus b\oplus a\oplus b, &&[a,b]_{\odot}=a^{-1}\odot b^{-1}\odot a\odot b.
\end{align*}

\section{Quotients where addition and multiplication coincide}
\begin{proposition}\label{firstproposition}Let $A$ be a skew brace, $B$ be a characteristic subgroup of $A_{\oplus}$, and $W$ be a set of group words. For $x\in A$ denote by $\overline{x}$ the image of $x$ under the natural homomorphism $A_{\oplus}\to A_{\oplus}/B$. If $W({\rm Aut}(A_{\oplus}/B))=1$, then for all $a,b\in W(A_{\odot})$ the equalities $\overline{a\odot b}=\overline{a\oplus b}$, $\overline{a^{-1}}=\overline{\ominus a}$ hold.
\end{proposition}
\begin{proof}
Consider the homomorphism $\psi:A_{\odot}\to {\rm Aut}(A_{\oplus}/B)$ which sends an element $a\in A_{\odot}$ to the automorphism 
$$\psi(a):t\oplus B\mapsto \lambda_a(t)\oplus B.$$
The map $\psi(a)$ is an automorphism of $A_{\oplus}/B$ since $\lambda_a$ is an automorphism of $A_{\oplus}$, and $B$ is characteristic in $A_{\oplus}$. The map $\psi:A_{\odot}\to {\rm Aut}(A_{\oplus}/B)$ is a homomorphism since the map $\lambda:A_{\odot}\to {\rm Aut}(A_{\oplus})$ given by $\lambda(a)=\lambda_a$ is a homomorphism, and $B$ is characteristic in $A_{\oplus}$. 

For every $a\in W(A_{\odot})$ the image $\psi(a)$ belongs to $W({\rm Aut}(A_{\oplus}/B))=1$, hence every $a\in W(A_{\odot})$ belongs to the kernel of $\psi$, where the kernel has the following form
\begin{align*}
{\rm ker}(\psi)&=\{a\in A_{\odot}~|~\psi(a)=id\}\\
&=\left\{a \in A_{\odot}~|~t \oplus B = \lambda_a(t) \oplus B~\text{for all}~t \in A\right\}\\
&=\{a\in A_{\odot}~|~\overline{\lambda_a(t)}=\overline{t}~\text{for all}~ t\in A\}.
\end{align*}
Therefore  due to formulas (\ref{multsum}), (\ref{inverseoposit}) for all $a,b\in W(A_{\odot})$ we have the equalities
\begin{align*}
\overline{a\odot b}=\overline{a\oplus\lambda_a(b)}=\overline{a\oplus b},&&\overline{a^{-1}}=\overline{\ominus\lambda_{a^{-1}}(a)}=\overline{\ominus a}.
\end{align*}
The statement is proved.
\end{proof}
\begin{corollary}\label{thefirstcorollary}Let $A$ be a skew brace, $B$ be a characteristic subgroup of $A_{\oplus}$ such that $A_{\oplus}^{(n)}\leq B$ for some $n$. If ${\rm Aut}(A_{\oplus}/B)$ is solvable of derived length $m$, then $A_{\odot}^{(m+n)}\leq B$.
\end{corollary}
\begin{proof}Every member of the derived series of a group is a verbal subgroup. Denote by $W$ the set of group words such that $W(G)=G^{(m+1)}$ for every group $G$. Since ${\rm Aut}(A_{\oplus}/B)$ is solvable of derived length $m$, the equality $W({\rm Aut}(A_{\oplus}/B))=({\rm Aut}(A_{\oplus}/B))^{(m+1)}=1$ holds. Therefore from Proposition~\ref{firstproposition} for all $a,b\in W(A_{\odot})=A_{\odot}^{(m+1)}$ we have the equalities 
\begin{align}
\label{weneedthisguys}
\overline{a\odot b}=\overline{a\oplus b},&& \overline{a^{-1}}=\overline{\ominus a},
\end{align} 
where $\overline{x}$ denotes the image of $x$ under the natural homomorphism $A_{\oplus}\to A_{\oplus}/B$. 

Using induction on $k$ let us prove that $\overline{A_{\odot}^{(m+k)}}\leq\overline{A_{\oplus}^{(k)}}$. The basis of induction ($k=1$) is obvious. Suppose that the statement is true for $k$, and let us prove it for $k+1$. From equalities~(\ref{weneedthisguys}) we conclude that for all $a_1,a_2,\dots,a_r,b_1,b_2,\dots,b_r\in A_{\odot}^{(m+k)}$ the equality 
$$\overline{[a_1,b_1]_{\odot}\odot [a_2,b_2]_{\odot}\odot\dots\odot[a_r,b_r]_{\odot}}=\overline{[a_1,b_1]_{\oplus}\oplus[a_2,b_2]_{\oplus}\oplus\dots\oplus[a_r,b_r]_{\oplus}}$$
holds. From this equality and the induction conjecture we conclude that $\overline{A_{\odot}^{(m+k+1)}}\leq\overline{A_{\oplus}^{(k+1)}}$. Hence $\overline{A_{\odot}^{(m+n)}}\leq\overline{A_{\oplus}^{(n)}}=\overline{0}$ and $A_{\odot}^{(m+n)}\leq B$.
\end{proof}

\section{Skew braces with a cyclic characteristic subgroup in the additive group}
For a positive integer $n$ we denote by $\mathbb{Z}_n$ both the cyclic group of order $n$ and the ring of residues modulo $n$. We denote by $\mathbb{Z}$ both the infinite cyclic group and the ring of integers. The following three lemmas give the information about automorphism groups of some abelian groups.
\begin{lemma}\label{zeroaut}Let $p$ be an odd prime, and $n$ be a positive integer. Then ${\rm Aut}(\mathbb{Z}_{p^{n}})=\mathbb{Z}_{p^{n - 1}(p - 1)}$.
\end{lemma}
\begin{proof} Denote by $x$ the generator of $\mathbb{Z}_{p^n}$, and let $\varphi$ be an automorphism of $\mathbb{Z}_{p^n}$. Then $\varphi(x)=kx$, where $k$ is coprime with $p$, and $\varphi$ is completely defined by $k$. If $\varphi_1,\varphi_2$ are two automorphisms of $\mathbb{Z}_{p^n}$ with $\varphi_1(x)=k_1x$, $\varphi_2(x)=k_2x$, then $\varphi_1\varphi_2(x)=k_1k_2x$. Therefore ${\rm Aut}(\mathbb{Z}_{p^n})$ is isomorphic to the group $\mathbb{Z}_{p^n}^*$ of invertible elements of the ring $\mathbb{Z}_{p^n}$ under the multiplication. Since for every odd $p$ there exists a primitive root modulo $p^n$, the group $\mathbb{Z}_{p^n}^*$ is cyclic and is isomorphic to $\mathbb{Z}_{p^{n-1}(p-1)}$.
\end{proof}
\begin{lemma}\label{firstaut} ${\rm Aut}(\mathbb{Z}_{2})=1$, ${\rm Aut}(\mathbb{Z}_{4})=\mathbb{Z}_2$, ${\rm Aut}(\mathbb{Z}_{2^n})=\mathbb{Z}_2\times \mathbb{Z}_{2^{n-2}}$ for $n\geq3$.
\end{lemma}
\begin{proof}The equalities ${\rm Aut}(\mathbb{Z}_{2})=1$, ${\rm Aut}(\mathbb{Z}_{4})=\mathbb{Z}_2$ are obvious. Let us prove that for $n\geq3$ the equality ${\rm Aut}(\mathbb{Z}_{2^n})=\mathbb{Z}_2\times \mathbb{Z}_{2^{n-2}}$ holds. Similar to Lemma~\ref{zeroaut} we conclude that ${\rm Aut}(\mathbb{Z}_{2^n})$ is isomorphic to the group $\mathbb{Z}_{2^n}^*$ of invertible elements of the ring $\mathbb{Z}_{2^n}$ under the multiplication. Since there are no primitive roots modulo $2^n$ for $n\geq3$, the group $\mathbb{Z}_{2^n}^*$ is not cyclic. 

Let us prove that $3$ is an element of order $2^{n-2}$ in $\mathbb{Z}_{2^n}^*$. Since $|\mathbb{Z}_{2^n}^*|=2^{n-1}$ and $\mathbb{Z}_{2^n}^*$ is not cyclic, we have $3^{2^{n-2}}=1$ in $\mathbb{Z}_{2^n}^*$, and it is enough to prove that $3^{2^{n-3}}\neq 1$ in $\mathbb{Z}_{2^n}^*$. Using induction on $k$ let us prove that $3^{2^{k}}\neq 1$ in $\mathbb{Z}_{2^{k+3}}^*$ for all $k$. The basis of induction ($k=1$) follows from the equality 
$$3^{2^{3-3}}=3\neq 1~({\rm mod}~8).$$
Suppose that the statement it true for some $k$, and let us prove it for $k+1$. We have the equality
$$
3^{2^{k+1}}-1=\left(3^{2^{k}}\right)^2-1=\left(3^{2^{k}}-1\right)\left(3^{2^{k}}+1\right).
$$
By the induction conjecture we conclude that $3^{2^{k}}-1\neq0$ in $\mathbb{Z}_{2^{k+3}}$. From this fact and the equality $$3^{2^{k}}+1=(-1)^{2^k}+1=2\neq 0~({\rm mod}~4)$$
we conclude that 
$$3^{2^{k+1}}-1=\left(3^{2^{k}}-1\right)\left(3^{2^{k}}+1\right)$$
is not equal to $0$ in $\mathbb{Z}_{2^{k+4}}$. Hence, $3^{2^{k}}\neq 1$ in $\mathbb{Z}_{2^{k+3}}^*$ for all $k$, and, in particular, $3^{2^{n-3}}\neq 1$ in $\mathbb{Z}_{2^{n}}^*$.

The group $\mathbb{Z}_{2^n}^*$ is abelian of order $2^{n-1}$, not cyclic, and it contains an element of order  $2^{n-2}$. The only group which satisfy this conditions is $\mathbb{Z}_2\times \mathbb{Z}_{2^{n-2}}$.
\end{proof}
\begin{lemma}\label{secondaut} For every positive integer $n$ the group ${\rm Aut}(\mathbb{Z}_2\times \mathbb{Z}_{2^{n}})$ is solvable.
\end{lemma}
\begin{proof}If $n=1$, then ${\rm Aut}(\mathbb{Z}_2\times \mathbb{Z}_{2^{n}})={\rm Aut}(\mathbb{Z}_2\times \mathbb{Z}_{2})=S_3$ is a metabelian group. If $n> 1$, then let us use \cite[Theorem 4.1]{HilRhe}, where the orders of automorphism groups of finite abelian groups are calculated. According to this theorem $|{\rm Aut}(\mathbb{Z}_2\times \mathbb{Z}_{2^{n}})|=2^{n+1}$. Therefore, ${\rm Aut}(\mathbb{Z}_2\times \mathbb{Z}_{2^{n}})$ is a $2$-group. Hence it is nilpotent and therefore is solvable.
\end{proof}

\begin{proposition}\label{secondproposition}Let $A$ be a skew brace, $D$ be a cyclic characteristic subgroup of $A_{\oplus}$, $B=C_{A_{\oplus}}(D)$. For $x\in A$ denote by $\overline{x}$ the image of $x$ under the natural homomorphism $A_{\oplus}\to A_{\oplus}/B$. Then there exists a positive integer $m$ such that  for all $a,b\in A_{\odot}^{(m)}$ the equalities $\overline{a\odot b}=\overline{a\oplus b}$, $\overline{a^{-1}}=\overline{\ominus} a$ hold.
\end{proposition}
\begin{proof}Let us first consider the case when $D$ is finite. Let $D=\langle x\rangle=\mathbb{Z}_n$ be a cyclic group of order $n$ generated by $x$. Let $n$ has the prime factors decomposition 
$$n=2^{\alpha_0}p_1^{\alpha_1}p_2^{\alpha_2}\dots p_s^{\alpha_s},$$
where $p_1,p_2,\dots,p_s$ are distinct odd primes, $\alpha_1,\alpha_2,\dots,\alpha_s$ are positive integers, and $\alpha_0$ is non-negative integer. 
For uniformity, denote by $p_0=2$. For $i=0,1,\dots,s$ denote by $D_i$ a subgroup of $D$ generated by the element $(n/p_i^{\alpha_i})x$. This group is cyclic of order $p_i^{\alpha_i}$. It is characteristic in $D$. Since $D$ is characteristic in $A_{\oplus}$, and $D_i$ is characteristic in $D$, the group $D_i$ is characteristic in $A_{\oplus}$. 

Fix $i\in \{0,1,\dots,s\}$ and consider the homomorphism $\varphi_i:A_{\oplus}\to {\rm Aut}(D_i)$ which sends an element $a\in A$ to the conjugation automorphism 
$$\varphi_i(a):y\mapsto \ominus a\oplus y\oplus a.$$
The map $\varphi_i(a)$ gives an automorphism of $B$ since the conjugation by $a$ in $A_{\oplus}$ is an automorphism of $A_{\oplus}$, and $D_i$ is characteristic in $A_{\oplus}$. The kernel of $\varphi_i$ has the form 
$${\rm ker}(\varphi_i)=\{a\in A_{\oplus}~|~\varphi_i(a)=id\}=\{a\in A_{\oplus}~|~\ominus a\oplus y\oplus a=y~\text{for all}~y\in D_i\}=C_{A_{\oplus}}(D_i).$$
Denote by $C_{A_{\oplus}}(D_i)=B_i$. By the homomorphism theorem we have the inclusion $A_{\oplus}/B_i\hookrightarrow{\rm Aut}(D_i)$. From Lemmas~\ref{zeroaut},~\ref{firstaut} it follows that $A_{\oplus}/B_i$ is either cyclic or of the form $\mathbb{Z}_2\times \mathbb{Z}_{2^k}$ for appropriate $k$. Therefore from Lemmas~\ref{zeroaut},~\ref{firstaut},~\ref{secondaut} it follows that ${\rm Aut}(A_{\oplus}/B_i)$ is solvable. Denote by $m_i$ the derived length of ${\rm Aut}(A_{\oplus}/B_i)$. 

From Proposition~\ref{firstproposition} it follows that  for all $a,b\in A_{\odot}^{(m_i)}$ the equalities $\overline{a\odot b}=\overline{a\oplus b}$, $\overline{a^{-1}}=\overline{\ominus a}$ hold, where $\overline{x}$ denotes the image of $x$ under the natural homomorphism $A_{\oplus}\to A_{\oplus}/B_i$. These equalities mean that for all $i=1,2,\dots,s$ the inclusions
\begin{align}\label{alottoone}(a\odot b)\ominus(a\oplus b)\in B_i,&&(a^{-1})\ominus (\ominus a)\in B_i\end{align}
hold for all $a,b\in A_{\odot}^{(m_i)}$. Denote by $m={\rm max}\{m_0,m_1,\dots,m_s\}$. Then from (\ref{alottoone}) it follows that for all $a,b\in A_{\odot}^{(m)}$ we have 
\begin{align*}(a\odot b)\ominus(a\oplus b)\in \bigcap_{i=1}^sB_i=B,&&(a^{-1})\ominus (\ominus a)\in \bigcap_{i=1}^sB_i=B.\end{align*}
Therefore for all $a,b\in A_{\odot}^{(m)}$ the equalities $\overline{a\odot b}=\overline{a\oplus b}$, $\overline{a^{-1}}=\overline{\ominus} a$ hold. 

The case when when $D$ is the infinite cyclic group is similar. \end{proof}

\begin{corollary}\label{cornice}Let $A$ be a skew brace, and $D$ be a cyclic characteristic subgroup of $A_{\oplus}$ such that $D\leq Z\left(A_{\oplus}^{(n)}\right)$ for some $n$. Then there exists a positive integer $m$ such that  $A_{\odot}^{(m+n)}\leq C_{A_{\oplus}}(D)$. 
\end{corollary}
\begin{proof}The proof is the same as the proof of Corollary~\ref{thefirstcorollary} for $B=C_{A_{\oplus}}(D)$ using Proposition~\ref{secondproposition} instead of Proposition~\ref{firstproposition} and taking into account that since $D\leq Z\left(A_{\oplus}^{(n)}\right)$, we must have the inclusion $A_{\oplus}^{(n)}\leq C_{A_{\oplus}}(D)$.\end{proof}
\section{Finite skew braces with a cyclic derived subgroup of the additive group}
The following theorem is proved in \cite{SmoVen}.
\begin{theorem}\label{SmoVenNilpotent} Let $A$ be a finite skew brace such that $A_{\oplus}$ is nilpotent. Then $A_{\odot}$ is solvable.
\end{theorem}

The following theorem gives a positive answer to the conjecture of A.~Smoktunowicz and L.~Vendramin in the case when $A_{\oplus}^{\prime}$ is a cyclic group.
\begin{theorem}\label{maintheorem}
Let $A$ be a finite skew brace such that $A_{\oplus}^{\prime}$ is cyclic. Then $A_{\odot}$ is solvable.
\end{theorem}
\begin{proof} For arbitrary elements $a,b,c\in C_{A_{\oplus}}(A_{\oplus}^{\prime})$ we have the equality $[[a,b]_{\oplus},c]_{\oplus}$. Therefore the group $C_{A_{\oplus}}(A_{\oplus}^{\prime})$ is nilpotent. This group is a characteristic subgroup of $A_{\oplus}$, therefore by Lemma~\ref{easylemma} it forms a brace with nilpotent additive group. From Theorem~\ref{SmoVenNilpotent} it follows that the multiplicative group of $C_{A_{\oplus}}(A_{\oplus}^{\prime})$ is solvable. From Corollary~\ref{cornice} it follows that there exists a positive integer $m$ such that 
$$A_{\odot}^{(m+2)}\leq C_{A_{\oplus}}(A_{\oplus}^{\prime}).$$
From this inclusion and the fact that the mutiplicative group of $C_{A_{\oplus}}(A_{\oplus}^{\prime})$ is solvable it follows that the multiplicative group of $A$ is solvable.
\end{proof}

{\small

\medskip

\medskip

\noindent
Baojun Li\\
School of Mathematics and Statistics, Nantong University, Jiangsu 226019, P. R. China\\
libj0618@outlook.com

~\\
Timur Nasybullov\\
School of Mathematics and Statistics, Nantong University, Jiangsu 226019, P. R. China\\
Novosibirsk State University, Pirogova 1, 630090 Novosibirsk, Russia\\
Sobolev Institute of Mathematics, Acad. Koptyug avenue 4, 630090 Novosibirsk, Russia\\
timur.nasybullov@mail.ru

~\\
Vyacheslav Zadvornov\\
Novosibirsk State University, Pirogova 1, 630090 Novosibirsk, Russia\\
v.zadvornov@g.nsu.ru

}
\end{document}